\documentclass{article}

\newcommand{\titel}{Division by Zero in Common Meadows\thanks{%
  The preceding version v2 (22 December 2014) 
  appeared in: Rocco De Nicola and Rolf Hennicker
  (eds.), \emph{Software, Services and Systems: Essays Dedicated 
  to Martin Wirsing}, LNCS 8950, pp. 46-61, Springer, 2015.
  Main changes: axiom~\eqref{Ma12} in Table~\ref{Mda} is new and 
  the proof of Thm.3.2.1 has been corrected.}}

\usepackage{stmaryrd,amssymb,amsmath,amsthm,url,accents,color}

\newtheorem{theorem}{Theorem}[subsection]
  
\newtheorem{proposition}[theorem]{Proposition}

\theoremstyle{definition}

\renewcommand{\bot}{\ensuremath{\textup{\textbf{a}}}}
\newcommand{\ga}{I}
\newcommand{\bota}{{\hat{\bot}}}

\newcommand{\inverse}{-}

\newcommand{\Md}{\mathsf{Md}}
\newcommand{\Mda}{\mathsf{Md}_\bot}
\newcommand{\CMCZ}{\mathsf{C}_0}
\newcommand{\CIL}{\mathsf{CIL}}

\newcommand{\AVL}{\mathsf{AVL}}
\newcommand{\FVL}{\mathsf{NVL}}
\newcommand{\NVL}{\FVL}
\newcommand{\VAR}{\mathsf{VAR}}
\newcommand{\iCAP}{\mathsf{ICL}}
\newcommand{\ICL}{\iCAP}
\newcommand{\CL}{\mathsf{CL}}

\newcommand{\CCM}{\textup{CCM}}
\newcommand{\CM}{\textup{CM}}
\newcommand{\CQCM}{\textup{CQCM}}

\newcommand{\Nat}{{\mathbb N}}
\newcommand{\Rat}{{\mathbb Q}}
\newcommand{\Int}{{\mathbb Z}}
\newcommand{\M}{{\mathbb M}}
\newcommand{\md}{{{md}}}

\newcommand{\lh}{\Bigl(}\newcommand{\rh}{\Bigl)}

\title{\titel}

\author{
	Jan A.\ Bergstra ~and~ Alban\ Ponse
	\\
\\
  {\small
	  Section Theory of Computer Science}\\[-1mm]
	  \small Informatics Institute, Faculty of Science\\[-1mm]
	  \small University of Amsterdam, The Netherlands\\[-.6mm]
	  \small \url{https://staff.fnwi.uva.nl/{j.a.bergstra/,a.ponse/}}
  }
\date{}

\begin{document}
\maketitle

\begin{abstract}
Common meadows are fields expanded with a total inverse function.  
Division by zero produces an 
additional value denoted with $\bot$ that propagates through all operations 
of the meadow signature (this additional value can be interpreted as an error element).
We provide a basis theorem for so-called common cancellation meadows of characteristic zero, 
that is, common meadows of characteristic zero that admit a certain cancellation law.
\\[4mm]
\end{abstract}

{\small\tableofcontents}

\section{Introduction}
\label{sec:Intro}
Elementary mathematics is uniformly taught around the world with a focus on
natural numbers, integers, fractions, and fraction calculation. 
The mathematical basis of that part of mathematics seems to reside in the field of 
rational numbers. 
In elementary teaching material the incorporation of rational numbers in a field is 
usually not made explicit. 
This leaves open the possibility that some other abstract datatype or some 
alternative abstract datatype specification improves upon fields in providing a 
setting in which such parts of elementary mathematics can be formalized.

In this paper we will propose the signature for | and model class of | 
\emph{common meadows} 
and we will provide a \emph{loose} algebraic specification of common meadows 
by way of a set of equations. 
In the terminology of Broy and Wirsing~\cite{BW81,Wir90}, the semantics of a loose 
algebraic specification $S$ is given by the class of all models of $S$, that is, the 
semantic approach is not restricted to the isomorphism class of initial algebras.
For a loose specification it is expected that its initial algebra is an important 
member of its model class, worth of independent investigation. 
In the case of common meadows this aspect is discussed in the last remark 
of Section~\ref{sec:4} (Concluding remarks).

A common meadow (using inversive notation) is an extension of a field equipped with a 
multiplicative inverse function $(...)^{\inverse 1}$ and an additional element 
$\bot$ that serves as the inverse of zero and propagates through all operations. 
It should be noticed that the use of the constant $\bot$ is a matter of convenience 
only because it merely constitutes a derived constant with defining equation 
$\bot = 0^{-1}$. 
This implies that all uses of $\bot$ can be removed from the story of common meadows 
(a further comment on this can be found in Section~\ref{sec:4}).

The inverse function of a common meadow is not an involution because 
$(0^{\inverse 1})^{\inverse 1}=\bot$.
We will refer to meadows with zero-totalized inverse, that is, $0^{\inverse 1}=0$,
as \emph{involutive meadows} because inverse becomes an involution. 
By default a ``meadow'' is assumed to be an involutive meadow.

The key distinction between meadows and fields, which we consider to be so 
important that it justifies a different name, is the presence of an operator symbol 
for inverse in the signature (inversive notation, see~\cite{BM2011a}) or for 
division (divisive notation, see~\cite{BM2011a}), 
where divisive notation $x/y$ is defined as $x\cdot y^{\inverse 1}$.
A major consequence is that fractions can be viewed as terms over the 
signature of (common) meadows. 
Another distinction between meadows and fields is that we do not require a meadow 
to satisfy the separation axiom $0\ne 1$.

The paper is structured as follows: below we conclude this section
with a brief introduction to some aspects of involutive meadows that will play
a role later on, and a discussion on why common meadows can be preferred over
involutive meadows. In Section~\ref{sec:2} we formally define common meadows and
present some elementary results. In Section~\ref{sec:3} we define ``common 
cancellation meadows'' and provide a basis theorem for common cancellation meadows 
of characteristic zero, which we consider our main result. Section~\ref{sec:4} 
contains some concluding remarks.

\subsection{Common Meadows versus Involutive Meadows}
\label{subsec:versus}
Involutive meadows, where instead of choosing $1/0 = \bot$, one calculates with 
$1/0 = 0$, 
constitute a different solution to the question how to deal with the value of 
$1/0$  once the design decision has been made to work with the signature of meadows, 
that is to include a function name for inverse or for division (or both) in
an extension of the syntax of fields. 
Involutive meadows feature a definite advantage over
common meadows in that, by avoiding an extension of the domain with an additional 
value, theoretical work is very close to classical algebra of fields. 
This conservation property, conserving the domain, of involutive meadows
has proven helpful for the development of theory about 
involutive meadows in~\cite{BBP2013a,BBP2013,BP2013,BM2011a,BR10,BT07}.
Earlier and comparable work on the equational 
theory of fields was done by Komori~\cite{K75} and Ono~\cite{Ono83}: in 1975, 
Komori introduced the name \emph{desirable pseudo-field} for what was introduced 
as a ``meadow'' in~\cite{BT07}.\footnote{\cite{BT07} was published in 2007; the 
  finding of~\cite{K75,Ono83} is mentioned 
  in~\cite{BM2011a} (2011) and was found via Ono's 1983-paper~\cite{Ono83}.}
  
An equational axiomatization $\Md$ of involutive meadows is given in 
Table~\ref{tab:Md}, where $^{\inverse 1}$ binds stronger than $\cdot$, which in 
turn binds stronger than $+$.
From the axioms in $\Md$ the following equations are derivable:
\begin{align*}
0\cdot x  &= 0,
&0^{-1}  &= 0,\\
x\cdot (-y) &= -(x\cdot y),
&(-x)^{-1} &= -(x^{-1}),
\\
-(-x) &= x,&(x \cdot  y)^{-1}
&= x^{-1} \cdot  y^{-1}.
\end{align*}

\begin{table}[t]
\centering
\hrule
\begin{align*}
(x+y)+z &= x + (y + z)
&x \cdot  y &= y \cdot  x\\
x+y &= y+x
&1\cdot x &= x\\
x+0 &= x
&x \cdot  (y + z) &= x \cdot  y + x \cdot  z\\
x+(-x) &= 0
&(x^{-1})^{-1} &= x \\
(x \cdot y) \cdot  z &= x \cdot  (y \cdot  z)
&x \cdot (x \cdot x^{-1}) &= x
\end{align*}
\hrule
\caption{The set $\Md$ of axioms for (involutive) meadows}
\label{tab:Md}
\end{table}

Involutive \emph{cancellation meadows} are involutive
meadows in which the following cancellation law holds:
\begin{equation}
\label{eq:CL}
\tag{$\CL$}
(x\ne 0 \wedge x\cdot y =x\cdot z) \rightarrow y=z.
\end{equation}
Involutive cancellation meadows form an important subclass of involutive meadows:
in~\cite[Thm.3.1]{BBP2013} it is shown that the axioms in 
Table~\ref{tab:Md} constitute a complete axiomatization of the equational theory 
of involutive cancellation meadows.
We will use a consequence of this result in Section~\ref{sec:3}.

A definite disadvantage of involutive meadows against common meadows is that 
$1/0 = 0$ is quite remote from common intuitions regarding the partiality of 
division. 

\subsection{Motivating a Preference for Common Meadows}
Whether common meadows are to be preferred over involutive meadows 
depends on the applications one may have in mind. 
We envisage as an application area the development of alternative foundations of 
elementary mathematics from a perspective of abstract datatypes, term rewriting, 
and mathematical logic. 
For that objective we consider common meadows to be the preferred option over 
involutive meadows. 
At the same time it can be acknowledged
that a systematic investigation of involutive meadows constitutes a necessary 
stage in the development of a theory of common meadows by facilitating in a 
simplified setting the determination of results which might be obtained about 
common meadows. Indeed each result about involutive meadows seems to suggest a 
(properly adapted) counterpart in the setting of common meadows, while proving or 
disproving such counterparts is not an obvious matter.

\section{Common Meadows}
\label{sec:2}
In this section we formally define ``common meadows'' by fixing their signature
and providing an equational axiomatization. Then, we consider some  conditional 
equations that follow from this axiomatization. 
Finally, we discuss some conditional laws that can be used to define
an important subclass of common meadows.
 
\subsection{Meadow Signatures}
The signature $\Sigma_{f}^S$ of fields (and rings) contains a sort (domain) $S$, 
two constants 0, and 1, two two-place functions $+$ (addition) and $\cdot$ 
(multiplication) and the one-place function $-$  (minus) for the inverse of addition.

We write $\Sigma_{\md}^S$ for the signature of meadows in inversive notation: 
\[\Sigma_{\md}^S = \Sigma_f^S \cup \{\_^{-1}\colon S \rightarrow S\},\]
and we write $\Sigma_{\md,\bot}^S$ for the signature of meadows in inversive notation 
with an $\bot$-totalized inverse operator: 
\[\Sigma_{\md,\bot}^S = \Sigma_{\md}^S \cup \{\bot \colon S\}.\]
The interpretation of $\bot$ is called the additional value and we write $\bota$ 
for this value.
Application of any function to the additional value returns that same value. 

When the name of the carrier is fixed it need not be mentioned explicitly in a 
signature. 
Thus, with this convention in mind, $\Sigma_{\md}$ represents $\Sigma_{\md}^S$ 
and so on.
If we want to make explicit that we consider terms over some signature 
$\Sigma$ with variables in set $X$, we write $\Sigma(X)$.

Given a field several meadow signatures and meadows can be connected with it. 
This will now be exemplified with the field $\Rat$ of rational numbers.
The following meadows are distinguished in this case:
\begin{description}
\item[$\Rat_0,$] the meadow of rational numbers with zero-totalized inverse: 
$\Sigma(\Rat_0) = \Sigma_{\md}^\Rat$.

\item[$\Rat_\bot,$] the meadow of rational numbers with $\bot$-totalized inverse:
$\Sigma(\Rat_\bot) = \Sigma_{\md,\bot}^{\Rat_a}$. 
The additional value $\bota$ interpreting $\bot$ has been taken outside $|\Rat|$ 
so that $|\Rat_{\bota}| = |\Rat| \cup \{\bota\}$.
\end{description}

\subsection{Axioms for Common Meadows}

\begin{table}[t]
\centering
\hrule
\begin{align}
\label{Ma1}
(x+y)+z &= x + (y + z)\\
\label{Ma2}
x+y     &= y+x\\
\label{Ma3}
x+0     &= x\\
\label{Ma4}
x+(-x)  &= 0 \cdot x \\
\label{Ma5}
(x \cdot y) \cdot  z &= x \cdot  (y \cdot  z)\\
\label{Ma6}
x \cdot  y &= y \cdot  x\\
\label{Ma7}
1\cdot x &= x \\
\label{Ma8}
x \cdot  (y + z) &= x \cdot  y + x \cdot  z\\[0mm]
\label{Ma9}
-(-x) &= x\\
\label{Ma10}
x \cdot x^{\inverse 1} &= 1  + 0 \cdot x^{\inverse 1}\\
\label{Ma11}
(x \cdot y)^{\inverse 1} &= x^{\inverse  1} \cdot y^{\inverse 1}\\
\label{Ma12}
(1+ 0 \cdot x)^{\inverse 1}&=1+ 0 \cdot x\\[0mm]
\label{Ma13}
0^{\inverse 1} &= \bot \\
\label{Ma14}
x + \bot &= \bot
\end{align}
\hrule
\caption{$\Md_\bot$, an independent set of axioms for common meadows}
\label{Mda}
\end{table}

The axioms in Table~\ref{Mda} define the class (variety) of \emph{common meadows}, 
where we adopt the convention that $\_^{\inverse 1}$ binds stronger than $\cdot$, 
which in turn binds stronger than $+$.
Some comments: Axiom~\eqref{Ma14} implies $\bot$'s propagation
through all operations, and for the same reason, axiom~\eqref{Ma10} has its
particular form. Axiom~\eqref{Ma4} is a variant of the 
common axiom on additional inverse, which also serves $\bot$'s propagation. 
Axioms~\eqref{Ma11} and \eqref{Ma12} are further 
equations needed for manipulation of $(...)^{-1}$-expressions. 
 
We note that the axiom set $\Mda$ is independent,
which can be easily demonstrated with the tool \emph{Mace4}, and that 
the typical
identities for common meadows established by
the following Propositions~\ref{prop:1} and~\ref{prop:2a}
were checked with the theorem prover 
\emph{Prover9}, see~\cite{McCune} for both these tools.

\begin{proposition}
\label{prop:1}
Equations that follow from $\Md_{\bot}$ (see Table~\ref{Mda}): 
\begin{align}
\label{e1}\tag{e1}
0\cdot 0&=0
\\
\label{e2}\tag{e2}
0 \cdot (x \cdot x) &= 0 \cdot x
\\
\label{e3}\tag{e3}
-0&=0,
\\
\label{e4}\tag{e4}
0\cdot x&=0\cdot (-x)
\\
\label{e5}\tag{e5}
0\cdot (x\cdot y)&=0\cdot (x+y)
\\
\label{e6}\tag{e6}
-(x\cdot y)&=x\cdot (-y)
\\
\label{e7}\tag{e7}
(-x)\cdot (-y)&=x\cdot y
\\
\label{e8}\tag{e8}
(-1)\cdot x&=-x
\\
\label{e9}\tag{e9}
1^{-1}&=1
\\
\label{e10}\tag{e10}
(x\cdot x^{\inverse 1})\cdot x^{\inverse 1}&=x^{\inverse 1}
\\
\label{e11}\tag{e11}
(-x)^{\inverse 1}&=-(x^{\inverse 1})
\\
\label{e12}\tag{e12}
(x \cdot x^{\inverse 1})^{\inverse 1}&=x \cdot x^{\inverse 1}
\\
\label{e13}\tag{e13}
(x^{\inverse 1})^{\inverse 1}&=x + 0 \cdot x^{\inverse 1}
\end{align}
and
\begin{align}
\label{e14}\tag{e14}
x \cdot \bot &= \bot
\\
\label{e15}\tag{e15}
-\bot&=\bot
\\
\label{e16}\tag{e16}
\bot^{\inverse 1}&=\bot
\end{align}
\end{proposition}

\begin{proof}
Most derivations are trivial. 
\begin{description}

\item[$\eqref{e1}.$]
By axioms~\eqref{Ma3}, \eqref{Ma7}, \eqref{Ma8}, \eqref{Ma2} we find
$x=(1+0)\cdot x=x+0\cdot x=0\cdot x + x$, hence 
$0=0\cdot 0+0$, so by axiom~\eqref{Ma3}, $0=0\cdot 0$.

\item[$\eqref{e2}.$]
First derive $0\cdot (x\cdot x)=(0\cdot x)\cdot(0\cdot x)$ and
$0\cdot x=0\cdot 1+0\cdot 0\cdot x=0\cdot(1+ 0\cdot x)$.
Hence, 
$0\cdot (x\cdot x)
=0\cdot (1+0\cdot x)\cdot(1+0\cdot x)
\stackrel{\eqref{Ma12}}=0\cdot (1+0\cdot x)\cdot(1+0\cdot x)^{\inverse 1}
\stackrel{\eqref{Ma10}}=0\cdot(1+0\cdot(1+0\cdot x)^{\inverse 1})
\stackrel{\eqref{Ma12}}=0\cdot(1+0\cdot(1+0\cdot x))
=0\cdot(1+0\cdot x)=0\cdot x$.

\item[$\eqref{e3}.$]
By axioms~\eqref{Ma3}, \eqref{Ma2}, \eqref{Ma4} and \eqref{e1} we find
$-0=(-0)+0=0+(-0)=0\cdot 0=0$.

\item[$\eqref{e4}.$]
By axioms~\eqref{Ma2}, \eqref{Ma4}, \eqref{Ma9} we find
$0\cdot x=x+(-x)=(-x)+-(-x)=0\cdot (-x)$.

\item[$\eqref{e5}.$]
First note $0\cdot x+0\cdot x=(0+0)\cdot x = 0\cdot x$.
By axioms~$\eqref{Ma2}-\eqref{Ma4}$, \eqref{Ma6}, \eqref{Ma8}, and~\eqref{e2}
we find
$0\cdot (x+y)=0\cdot ((x+y)\cdot(x+y))
=(0\cdot x+0\cdot(x\cdot y))+(0\cdot y+0\cdot(x\cdot y))
=(0+ 0\cdot y)\cdot x+(0+0\cdot x)\cdot y=0\cdot(x\cdot y)+0\cdot(x\cdot y)=
0\cdot(x\cdot y)$.

\item[$\eqref{e6}.$] We give a detailed derivation:
\begin{align*}
-(x\cdot y)&=-(x\cdot y)+0\cdot -(x\cdot y)
&&\text{by $x=x+0\cdot x$}\\
&=-(x\cdot y)+0\cdot(x\cdot y)
&&\text{by \eqref{e4}}\\
&=-(x\cdot y)+x\cdot(0\cdot y)
&&\text{by axioms \eqref{Ma5} and \eqref{Ma6}}\\
&=-(x\cdot y)+x\cdot(y+(-y))
&&\text{by axiom \eqref{Ma4}}\\
&=-(x\cdot y)+(x\cdot y+x\cdot(-y))
&&\text{by axiom \eqref{Ma8}}\\
&=(-(x\cdot y)+x\cdot y)+x\cdot(-y)
&&\text{by axiom \eqref{Ma1}}\\
&=0\cdot(x\cdot y)+x\cdot(-y)
&&\text{by axioms \eqref{Ma2} and \eqref{Ma4}}\\
&=0\cdot(x\cdot -y)+x\cdot(-y)
&&\text{by axioms \eqref{Ma6} and \eqref{Ma5}, and \eqref{e4}}\\
&=x\cdot(-y).
&&\text{by $x=0\cdot x+x$}
\end{align*}

\item[$\eqref{e7}.$]
By \eqref{e6},
$(-x)\cdot (-y)=-((-x)\cdot y)=-(y\cdot (-x))=-(-(y\cdot x))=x\cdot y$.

\item[$\eqref{e8}.$]
From \eqref{e6} with $y=1$ we find $-x=-(x\cdot 1)=x\cdot (-1)=(-1)\cdot x$.

\item[$\eqref{e9}.$]
By~\eqref{e1} and axioms~\eqref{Ma3} and \eqref{Ma12},
$1^{-1}=(1+0\cdot0)^{-1}=1+0\cdot0=1$.

\item[$\eqref{e10}.$]
By axioms~\eqref{Ma10} and~\eqref{e2}, $(x\cdot x^{\inverse 1})\cdot x^{\inverse 1}=
(1+0\cdot x^{\inverse 1})\cdot x^{\inverse 1}=x^{\inverse 1}+0\cdot x^{\inverse 1}
=x^{\inverse 1}$.

\item[$\eqref{e11}.$]
By \eqref{e10} and \eqref{e7}, $(-1)^{-1}=-1\cdot (-1)^{-1}\cdot(-1)^{-1}=-1\cdot 
((-1)\cdot(-1))^{-1}=-1\cdot 1^{-1}=-1$.
Hence, $(-x)^{-1}=(-1\cdot x)^{-1}=(-1)^{-1}\cdot x^{-1}=-1\cdot x^{-1}
=-(x^{-1})$.

\item[$\eqref{e12}.$]
By axioms~\eqref{Ma10} and~\eqref{Ma12},
$x \cdot x^{\inverse 1} = 1 + 0 \cdot x^{\inverse 1} = (1 + 0\cdot 
x^{\inverse 1})^{\inverse1} = (x \cdot x^{\inverse 1})^{\inverse 1}$.

\item[$\eqref{e13}.$]
By~\eqref{e10}, $(x^{\inverse 1})^{\inverse1} =x^{\inverse1} \cdot  
(x^{\inverse1})^{\inverse1} \cdot (x^{\inverse1})^{\inverse1} 
\stackrel{\eqref{Ma11}}=
(x \cdot  x^{\inverse1})^{\inverse1} \cdot (x^{\inverse1})^{\inverse1}
\stackrel{\eqref{e12}}=
(x \cdot  x^{\inverse1}) \cdot (x^{\inverse1})^{\inverse1} = x \cdot (x \cdot 
x^{\inverse1})^{\inverse1} 
\stackrel{\eqref{e12}}= x \cdot  (x \cdot x^{\inverse1}) = x\cdot (1 + 0\cdot 
x^{\inverse1}) =x + 0\cdot x\cdot x^{\inverse 1} =
 x + 0\cdot ( 1 + 0 \cdot x^{\inverse 1}) = x+ 0 \cdot x^{\inverse1}.$

\item[$\eqref{e14}.$]
By axioms~\eqref{Ma8} and~\eqref{Ma14}, $\bot\cdot(1+x)=\bot+\bot\cdot x=\bot$, hence 
$\bot\cdot x=\bot\cdot((1+(-1))+x)=\bot\cdot(1+(-1+x))=\bot$, and thus 
$x\cdot\bot=\bot$ by axiom~\eqref{Ma6}.

\item[$\eqref{e15}.$]
By axioms~\eqref{Ma6} and~\eqref{Ma5}, and~\eqref{e6} and~\eqref{e14}, 
$-\bot=-(\bot\cdot 1)=\bot\cdot(-1)=\bot$.

\item[$\eqref{e16}.$]
By axioms~\eqref{Ma13} and~\eqref{Ma11},  and \eqref{e14},
$\bot^{\inverse 1}=(0 \cdot \bot) ^{\inverse 1} =0 ^{\inverse 1} \cdot 
\bot^{\inverse 1}= \bot \cdot \bot^{\inverse 1}=\bot$.
\end{description}
\end{proof}

The next proposition establishes a generalization of a familiar identity concerning  
the addition of fractions.

\begin{proposition}
\label{prop:2a}
$\Mda\vdash \displaystyle x \cdot y^{\inverse  1}  + u \cdot v^{\inverse  1}
= (x \cdot v + u \cdot y) \cdot (y \cdot v)^{\inverse  1}$.
\end{proposition}
 
\begin{proof} We first derive 
\begin{align}
\nonumber
x \cdot y \cdot y^{\inverse  1}
&=x\cdot(1+0\cdot y^{\inverse  1})
&&\text{by axiom~\eqref{Ma10}}\\
\nonumber
&=x+0\cdot x\cdot y^{\inverse  1}\\
\nonumber
&=x+0\cdot x+0\cdot y^{\inverse  1}
&&\text{by~\eqref{e5}}\\
\label{eq:cv}
&=x+0\cdot y^{\inverse  1}.
\end{align}
Hence,
\begin{align*}
(x \cdot v + u \cdot y) \cdot (y \cdot v)^{\inverse  1}&=
x \cdot y^{\inverse  1} \cdot v \cdot v^{\inverse  1}+
u \cdot v^{\inverse  1} \cdot y \cdot y^{\inverse  1}\\
&=(x \cdot y^{\inverse  1}+0\cdot v^{\inverse  1})+
(u \cdot v^{\inverse  1} +0 \cdot y^{\inverse  1})
&&\text{by \eqref{eq:cv}}\\
&=(x \cdot y^{\inverse  1}+0\cdot y^{\inverse  1})+
(u \cdot v^{\inverse  1} +0 \cdot v^{\inverse  1})\\
&=x \cdot y^{\inverse  1}+u \cdot v^{\inverse  1}.
\end{align*}
\end{proof}

We end this section with two more propositions that characterize typical
properties of common meadows and that are used in the proof of 
Theorem~\ref{thm:basis}. The first of these establishes that each 
(possibly open) 
term over $\Sigma_{\md,\bot}$ has a simple representation in the syntax of meadows.
 
\begin{proposition}
\label{prop:3}
For each term $t$ over $\Sigma_{\md,\bot}(X)$ with variables in $X$
there exist terms $r_1,r_2$ over $\Sigma_f(X)$ such that 
$\Md_\bot\vdash t=r_1\cdot r_2^{-1}$ and $\VAR(t)=\VAR(r_1)\cup\VAR(r_2)$.
\end{proposition}

\begin{proof}
By induction on the structure of $t$, where the $\VAR(t)$-property follows 
easily in each case.

\begin{description}
\item[\normalfont If $t\in\{0,1,x,\bot\}$,]this follows trivially
(for the first three cases use $1^{-1}=1$).

\item[\normalfont Case $t\equiv t_1+t_2$.]
By Proposition~\ref{prop:2a}.
 
\item[\normalfont Case $t\equiv t_1\cdot t_2$.] Trivial. 

\item[\normalfont Case $t\equiv-t_1$.] By Proposition~\ref{prop:1}  (\ref{e6}).

\item[\normalfont Case $t\equiv t_1^{\inverse 1}$.] By induction there exist
$r_i\in \Sigma_f(X)$ such that $\Md_\bot\vdash t_1=r_1\cdot r_2^{\inverse 1}$. 
Now derive 
$t_1^{\inverse 1}=r_1^{\inverse 1}\cdot(r_2^{\inverse 1})^{\inverse 1}
=r_1^{\inverse 1}\cdot(r_2+0\cdot r_2^{\inverse 1})
=r_2\cdot r_1^{\inverse 1}+0\cdot r_1^{\inverse 1}+ 0\cdot r_2^{\inverse 1}
=r_2\cdot r_1^{\inverse 1}+ 0\cdot r_2^{\inverse 1}
$ and apply Proposition~\ref{prop:2a}.
\end{description}
\end{proof}

\newpage
The next proposition shows how a term of the form $0\cdot t$ with $t$ a
(possibly open) term over $\Sigma_f(X)$ can  be simplified (note that 
$0\cdot x = 0$ is not valid, since $0\cdot\bot=\bot$).
\begin{proposition}
\label{prop:3a}
For each term $t$ over $\Sigma_f(X)$, 
$\Md_\bot\vdash 0\cdot t=0\cdot \sum_{x\in\VAR(t)}x$, where 
$\sum_{x\in\emptyset}x=0$.
\end{proposition}

\begin{proof}
By induction on the structure of $t$, where equation~\eqref{e5} 
(Proposition~\ref{prop:1}) covers the multiplicative case.
\end{proof}
 
\subsection{Conditional Equations}
We discuss a number of conditional equations that will turn
out useful,
and we start off with a few that follow directly from $\Mda$.

\begin{proposition}
\label{prop:2}
Conditional equations that follow from $\Md_{\bot}$ (see Table~\ref{Mda}): 
\begin{align}
\label{ce1}\tag{ce1}
x \cdot y = 1 &\rightarrow 0 \cdot y = 0,
\\
\label{ce2}\tag{ce2}
x \cdot y = 1 &\rightarrow x^{\inverse  1} = y,
\\
\label{ce25}\tag{ce3}
0\cdot x=0\cdot y&\rightarrow 0 \cdot (x\cdot y) = 0\cdot x,
\\
\label{ce3}\tag{ce4}
0 \cdot x \cdot y  = 0 &\rightarrow 0 \cdot x = 0,
\\
\label{ce4}\tag{ce5}
0 \cdot (x + y ) = 0 &\rightarrow 0 \cdot x = 0,
\\
\label{ce5}\tag{ce6}
0 \cdot x^{\inverse 1} =0 &\rightarrow 0 \cdot x = 0,
\\
\label{ce7}\tag{ce7}
0 \cdot x = \bot &\rightarrow x = \bot.
\end{align}
\end{proposition}

\begin{proof}
Most derivations are trivial. 

\begin{description}

\item[$\eqref{ce1}.$]
By equations~\eqref{e2} and \eqref{e5},
$0\cdot x\cdot y=
0\cdot x\cdot y\cdot y
=0\cdot x\cdot y+0\cdot y\cdot y
=(0\cdot x+0\cdot y)\cdot y$,
and hence by assumption,
$0 = 0 \cdot 1 = 0 \cdot x \cdot y = (0 \cdot x + 0 \cdot y) \cdot y =
 0 \cdot x  \cdot y + 0 \cdot y \cdot y = 0 + 0 \cdot y = 0 \cdot y$.

\item[$\eqref{ce2}.$]
By assumption and axioms~\eqref{Ma11} and \eqref{Ma12}, 
$x^{\inverse 1}\cdot y^{\inverse 1}=1$, and thus by \eqref{ce1}, 
$0\cdot x^{\inverse 1}=0$, so by axiom~\eqref{Ma10}, 
$y=(1+0\cdot x^{\inverse 1})\cdot y=(x\cdot x^{\inverse 1})\cdot y=
(x\cdot y)\cdot x^{\inverse 1}=x^{\inverse 1}$.

\item[$\eqref{ce25}.$]
By assumption, equation~\eqref{e5}, and axiom~\eqref{Ma8}, 
$0\cdot(x\cdot y)=0\cdot x+0\cdot y=0\cdot x + 0\cdot x=0\cdot x$.

\item[$\eqref{ce3}.$]
By assumption, 
$0\cdot x=0\cdot x+0\cdot x\cdot y=x\cdot(0+0\cdot y)=0\cdot(x\cdot y)=0$.

\item[$\eqref{ce4}.$]
Apply equation~\eqref{e5} to \eqref{ce3}.

\item[$\eqref{ce5}.$]
By axiom~\eqref{Ma10} and assumption, 
$x\cdot x^{\inverse 1}=1+0\cdot x^{\inverse 1}=1$, 
so by~\eqref{ce1}, $0\cdot x=0$.

\item[$\eqref{ce7}.$]
By $x=x+0\cdot x$ and assumption, $x = x + \bot = \bot$.
\end{description}
\end{proof}

Note that~\eqref{ce1} and \eqref{ce2} immediately imply
\begin{equation*}
x \cdot y = 1 \rightarrow 0 \cdot x^{\inverse  1} = 0.
\end{equation*}

In Table~\ref{tab:laws} we define various conditional laws that we will use 
to single out certain classes of common meadows in Section~\ref{sec:3}: 
the Normal Value Law ($\NVL$), the Additional Value Law 
($\AVL$), and  the Common Inverse Law ($\CIL$).
Here we use the adjective ``normal'' to express that values different from $\bot$ 
(more precisely, the interpretation of $\bot$) are at stake.
We conclude this section by interrelating these laws.

\begin{table}[t]
\centering
\hrule
\begin{align*}
\tag{$\NVL$}
x \neq \bot& ~\rightarrow~ 0\cdot x =0
&&\text{Normal Value Law}\\
\tag{$\AVL$}
x^{\inverse 1}= \bot& ~\rightarrow~ 0\cdot x =x
&&\text{Additional Value Law}\\
\tag{$\CIL$}
x \neq 0 \wedge x \neq \bot& ~\rightarrow~ x\cdot x^{-1} =1
&&\text{Common Inverse Law}
\end{align*}
\hrule
\caption{Some conditional laws for common meadows}
\label{tab:laws}
\end{table}

\begin{proposition}
\label{prop:4}~ 
\begin{enumerate}
\item[$1.$]
\label{4.1}
$\Md_\bot+\NVL\vdash (x \cdot y = \bot \wedge x \neq \bot) \rightarrow   y = \bot$,

\item[$2.$]
\label{4.2}
$\Md_\bot+\NVL\vdash x^{\inverse 1}  \neq \bot \rightarrow 0 \cdot x = 0$,

\item[$3.$] 
\label{4.3}
$\Md_\bot+\NVL+\AVL\vdash \CIL$,

\item[$4.$]$\Md_\bot+\CIL\vdash\NVL$,

\item[$5.$]
$\Mda+\CIL\vdash \AVL$.
\end{enumerate}
\end{proposition}

\begin{proof}~
\begin{enumerate}
\item
By $\NVL$,
$x\ne \bot\rightarrow 0\cdot x=0$, so $0\cdot y=(0\cdot x)\cdot y=
0\cdot (x\cdot y)=0\cdot\bot=\bot$ 
and hence $y=(1+0)\cdot y=y+0\cdot y=y+\bot=\bot$.

\item
By $\NVL$, $0\cdot x^{\inverse 1}=0$ and hence by axiom~\eqref{Ma10},
$x\cdot x^{\inverse 1}=1$ and by~\eqref{ce1}, $0 \cdot x= 0$.

\item
From $x \neq \bot$ we find $0 \cdot x = 0$. 
There are two cases: $ x^{\inverse 1} = \bot$ which implies by $\AVL$
that $x= 0$  contradicting the assumptions of $\CIL$, and 
$ x^{\inverse 1} \neq \bot$ which implies by $\NVL$ that $0 \cdot x^{\inverse 1}=0$,
and this implies  $x \cdot x^{\inverse 1} =1$ by axiom~\eqref{Ma10}.

\item
Assume that $x \neq \bot$. If $x = 0$ then also $0 \cdot x=0$.
If $x \neq 0$ then by $\CIL$, 
$x \cdot x^{\inverse 1} = 1$, so
$0 \cdot x= 0$ by~\eqref{ce1}.

\item
We distinguish three cases: $x=0$, $x=\bot$, and $x \neq 0 \wedge x \neq \bot$. 
In the first two cases it immediately follows that $0\cdot x=x$. 
In the last case it follows by 
$\CIL$ that $x\cdot x\cdot x^{\inverse 1}=x$, so $x^{\inverse 1}=\bot$ implies  
$x=\bot$, and thus $x=0\cdot x$.

\end{enumerate}
\end{proof}

\section{Models and Model Classes}
\label{sec:3}
In this section we define ``common cancellation meadows'' as common meadows that 
satisfy the so-called ``inverse cancellation law'', a law that is equivalent with 
the Common Inverse Law $\CIL$. 
Then, we provide a basis theorem for common cancellation meadows of 
characteristic zero. 

\subsection{Common Cancellation Meadows}
In~\cite[Thm.3.1]{BBP2013} we prove a generic basis theorem that implies that the 
axioms in Table~\ref{tab:Md} constitute a complete axiomatization of the equational 
theory of the involutive cancellation meadows (over signature $\Sigma_{\md}$). 
The cancellation law used in that result (that is, \ref{eq:CL} in 
Section~\ref{subsec:versus}) has various 
equivalent versions, and a particular one is $x\ne 0 \rightarrow x\cdot x^{-1}=1$, 
a version that is close to $\CIL$.

Below we define common cancellation meadows, using a cancellation law that is 
equivalent with $\CIL$, but first we establish a correspondence between models of 
$\Md_\bot+\NVL+\AVL$ and involutive cancellation meadows.

\begin{proposition}
\label{prop:7}~~
\begin{itemize}
\item[$1.$]
\label{7.1}
Every field can be extended with an additional value $\bota$ and subsequently it 
can be expanded with a constant $\bot$ and an inverse function in such a way that 
the equations of common meadows as well as $\NVL$
and $\AVL$ are satisfied, where the interpretation of $\bot$ is $\bota$.

\item[$2.$]
\label{7.2}
A model of $\Md_\bot+\NVL+\AVL$ extends a field with an additional value $\bota$ (the
interpretation of $\bot$) and expands it with the $\bot$-totalized inverse. 
\end{itemize}
\end{proposition}

\begin{proof}
Statement 1 follows immediately. To prove 2,
consider the substructure of elements $b$ of the domain that satisfy $0 \cdot b = 0$. 
Only $\bota$ is outside this
subset. For $b$ with $0 \cdot b = 0$ we must check that $0 \cdot b^{\inverse 1} = 0$
unless $b=0$. 
To see this distinguish two cases: $b^{\inverse 1} = \bot$ 
(which implies $b=0$ with help of $\AVL$),
and $b^{\inverse 1} \neq \bot$ which implies $0 \cdot b^{\inverse 1} = 0$ by $\NVL$.
\end{proof}

As a consequence, we find the following result.
\begin{theorem}
\label{thm:one-one}
The models of $\Md_\bot+\NVL+\AVL$ that satisfy $0\ne 1$ are in one-to-one 
correspondence with the involutive cancellation meadows 
satisfying  $\Md$ (see Table~\ref{tab:Md}).
\end{theorem}

\begin{proof}
An involutive cancellation meadow can be expanded to a model of $\Md_\bot+\NVL+\AVL$ 
by extending its domain with a constant $\bota$ in such a way that the equations of 
common meadows as well as $\NVL$ and $\AVL$ are satisfied, where the interpretation 
of $\bot$ is $\bota$ (cf.~Proposition~\ref{prop:7}.1).

Conversely, given a model $\M$ of $\Md_\bot+\NVL+\AVL$, we  construct a cancellation 
meadow $\M'$ as follows:
$|\M'|=|\M|\setminus\{\bota\}$ with $\bota$ the interpretation of $\bot$, and 
$0^{-1}=0$ (by $0\ne 1$, $|\M'|$ is non-empty). 
We find by $\NVL$ that $0\cdot x=0$ and by 
$\CIL$ (thus by $\NVL+\AVL$, cf.~Proposition~\ref{prop:4}.3) that
$x\ne 0\rightarrow x\cdot x^{-1}=1$, which shows that $\M'$ is a cancellation meadow. 
\end{proof}

We define a \emph{common cancellation meadow} as a common meadow that satisfies 
the following \emph{inverse cancellation law} \eqref{CL}:
\begin{equation}
\label{CL}
\tag{$\iCAP$}
(x \neq 0 \wedge x \neq \bot \wedge x^{\inverse 1}\cdot y = x^{\inverse 1}\cdot z)
\rightarrow y=z.
\end{equation}
The class $\CCM$ of common cancellation meadows is axiomatized by $\Mda+\CIL$
in Table~\ref{Mda} and Table~\ref{tab:laws}, respectively. 
In combination with $\Mda$, the laws $\ICL$ and $\CIL$ are equivalent: 
first, $\Mda + \ICL \vdash \CIL$ because 
\[
(x \neq 0 \wedge x \neq \bot)\stackrel{\eqref{e10}}\rightarrow 
(x \neq 0 \wedge x \neq \bot\wedge x^{\inverse 1}\cdot x \cdot x^{\inverse 1}
= x^{\inverse 1} \cdot 1)\stackrel{\ICL}\rightarrow x\cdot x^{\inverse 1}=1.
\]
Conversely,
$\Mda+\CIL\vdash\ICL$:
\[
(x \neq 0 \wedge x \neq \bot \wedge x^{\inverse 1}\cdot y 
= x^{\inverse 1}\cdot z)\rightarrow x\cdot x^{\inverse 1}\cdot y=
x\cdot x^{\inverse 1} \cdot z\stackrel{\CIL}\rightarrow y=z.
\]

\subsection{A Basis Theorem for Common Cancellation Meadows of Characteristic Zero}

\begin{table}[t]
\centering
\hrule
\begin{align*}
\underline{n+1}\cdot (\underline{n+1})^{-1} &=1&&\text{($n\in\Nat$)}
\tag{$\CMCZ$}\\[2mm]
\underline{0}&=0&&\text{(axioms for }\\
\underline{1}&=1&&\text{~numerals, }\\
\underline{n+1}&=\underline{n} + 1&&\text{~$n\in\Nat$ and $n\geq 1$})
\end{align*}
\hrule
\caption{$\CMCZ$, the set of axioms for meadows of characteristic zero and numerals}
\label{CMCZ}
\end{table}

As in our paper~\cite{BBP2013a}, we use numerals $\underline{n}$ 
and the axiom scheme $\CMCZ$ defined in Table~\ref{CMCZ} to single out common
cancellation meadows of characteristic zero. In this section we prove that
$\Md_\bot+\CMCZ$ constitutes an axiomatization for common cancellation meadows of 
characteristic zero. In~\cite[Cor.2.7]{BBP2013a} we prove that $\Md+\CMCZ$ 
(for $\Md$ see Table~\ref{tab:Md}) constitutes 
an axiomatization for involutive cancellation meadows of characteristic zero.
We define $\CCM_0$ as the class of common cancellation meadows of characteristic 
zero.

\medskip

We further write 
$~\dfrac{t}{r}~$ (and sometimes
$t/r$ in plain text) ~for $~t\cdot r^{\inverse 1}$.

\begin{theorem}
\label{thm:basis}
$\Md_\bot+\CMCZ$ is a basis for the equational theory of $\CCM_0$.
\end{theorem}

\begin{proof}
Soundness holds by definition of $\CCM_0$.
 
In order to prove completeness, we consider two cases.
\\[2mm]
\underline{Case 1}.
Assume $\CCM_0\models t=\bot$. By Proposition~\ref{prop:3}
we can bring $t$ in the form $t_1/t_2$ with $t_1,t_2$ polynomials over $\Sigma_f(X)$, 
thus
\begin{equation}
\label{v00}
\CCM_0\models t=\frac{t_1}{t_2}.
\end{equation}
Then in each model $\M\in\CCM_0$, $t_1=\bot$ or $t_2=\bot$ or $t_2=0$. By definition of $\CCM_0$,
the first two cases are excluded because in $\M$
each variable of both $t_1$ and $t_2$ can be interpreted as a non-$\hat{\bot}$ value (e.g. 0), 
and then $t_1$ ($t_2$) evaluates to a value different from $\hat{\bot}$ in $\M$.
Thus~\eqref{v00} implies that $t_2=0$ holds in each field. 
Apparently it holds in each field that 
\[\textstyle
t_2=0\cdot \sum_{x\in\VAR(t_2)}x
\]
and it must be the case that in the polynomial $t_2$, all coefficients are 0. 
Furthermore, if for a closed term $s$ over $\Sigma_f$  
it holds that $s=0$ in each field,
then there exists an equational proof of $s=0$ in which substitutions happen first and 
only closed instances of the axioms for commutative rings are 
used.\footnote{%
  Without loss of generality it can be assumed that in equational proofs,
  substitutions happen first,
  see, e.g.~\cite{Groote}.} 
This implies that $\Md_\bot\vdash s=0$ because by Proposition~\ref{prop:3a}, $\Md_\bot\vdash 0\cdot s=0$ 
and thus each application of $\Md_\bot$-axiom~\eqref{Ma4}
(that is $x+(-x)=0\cdot x$), reduces to $u+(-u)=0$ with $u$ a closed term over $\Sigma_f$. 

It follows that $\Md_\bot\vdash t_2=0$, so by axiom~\eqref{Ma13} and equation~\eqref{e14}, 
$\Md_\bot\vdash t=t_1/t_2=\bot$.
\\[2mm]
\underline{Case 2}.
Assume $\CCM_0\models t=r$ and $\CCM_0\not\models t=\bot$ (and thus $\CCM_0\not\models r=\bot$). 
By Proposition~\ref{prop:3} we can bring $t$ in the form $t_1/t_2$
and $r$ in the form $r_1/r_2$ with $t_i,r_i$ polynomials over 
$\Sigma_f(X)$, thus
\begin{equation}
\label{v0}
\CCM_0\models \frac{t_1}{t_2}=\frac{r_1}{r_2}.
\end{equation}
Using axiom~\eqref{Ma12} it can be guaranteed that in $t_2$ and 
\label{pagekey}
$r_2$ no summands with coefficient $0$ occur.\footnote{%
\label{vn}
  This was the reason to revise the preceding version of this 
  paper (v3), in which $\Mda$ is defined differently: axiom~\eqref{Ma12} 
  was replaced by $1^{-1}=1$
  and equations~\eqref{e2} and~\eqref{e13} (Prop.\ref{prop:1}) were 
  also axioms; this resulted in a weaker system from which 
  axiom~\eqref{Ma12} cannot be derived, whereas the 
  absence of summands with coefficient $0$ in $t_2$ and 
  $r_2$ can be proven with the current modification of $\Mda$. 
  For example, using $x + 0\cdot y = x\cdot (1 + 0\cdot y)$ 
  (which follows easily),
  $(3x+2) \cdot  (6x^5 + 0\cdot x y + 2yz^3 + 4)^{-1} =
  (3x+2)\cdot ((6x^5 + 2yz^3 + 4)\cdot (1 + 0\cdot x y))^{-1}= 
  ((3x+2)\cdot (1 + 0\cdot x y))\cdot (6 x^5 + 2yz^3 + 4)^{-1}$.
}
We will first argue that \eqref{v0} implies that the following three equations 
are valid in $\CCM_0$:
\begin{align}
\label{p1}
0\cdot t_2^{\inverse 1}&= 0\cdot r_2^{\inverse 1},\\
\label{p2}
0\cdot t_1+0\cdot t_2&=0\cdot r_1+0\cdot r_2,\\
\label{p3}
t_2\cdot r_2\cdot(t_1\cdot r_2+ (-r_1)\cdot t_2)+0\cdot t_2^{\inverse 1}+0\cdot r_2^{\inverse 1}
&=0\cdot t_1+0\cdot t_2^{\inverse 1}+0\cdot r_1+0\cdot r_2^{\inverse 1}.
\end{align}
\underline{Ad~\eqref{p1}}. 
Assume this is not the case, then there exists a common cancellation meadow 
$\M\in\CCM_0$ and an interpretation of the variables in $t_2$ and $r_2$ such that 
one of $\smash{t_2^{\inverse 1}}$ and $\smash{r_2^{\inverse 1}}$ is interpreted as 
$\bota$ (the interpretation of $\bot$), and the other is not. 
This contradicts~\eqref{v0}.
\\[2mm]
\underline{Ad~\eqref{p2}}. 
This equation characterizes that $t_1/t_2$ and $r_1/r_2$ contain the same
variables, and is related to Proposition~\ref{prop:3a}.
Assume this is not the case, say $t_1$ and/or $t_2$ contains a variable $x$ that 
does not occur in $r_1$ and $r_2$. Since $\CCM_0\not\models r_1/r_2=\bot$, there is 
an instance of $r_i$'s variables, say $\overline{r_i}$ such that 
$\CCM_0\models \overline{r_1}/\overline{r_2}\ne\bot$.
But then $x$ can be instantiated with $\bot$, which contradicts~\eqref{v0}.
\\[2mm]
\underline{Ad~\eqref{p3}}. It follows from~\eqref{v0} that in~\eqref{p3}
both the lefthand-side and the right\-hand-side equal zero in all involutive 
cancellation meadows.
By Theorem~\ref{thm:one-one} we find $\CCM\models \eqref{p3}$, and hence 
$\CCM_0\models \eqref{p3}$.

\medskip

We now argue that~$\eqref{p1} - \eqref{p3}$
are derivable from $\Mda + \CMCZ$, and that from those~\eqref{v0} is derivable 
from $\Mda + \CMCZ$.
\\[2mm]
\underline{Ad~\eqref{p1}}. The statement 
$\CCM_0\models 0\cdot t_2^{\inverse 1}= 0\cdot r_2^{\inverse 1}$
implies that $t_2$ and $r_2$ have the same
zeros in the algebraic closure $\smash{\overline{\Rat}}$ of $\Rat$ 
(if this were not the case, then
$\overline{\Rat}_\bot\not\models0\cdot t_2^{\inverse 1}= 0\cdot r_2^{\inverse 1}$,
but $\overline{\Rat}_\bot\in \CCM_0$). 
By the Nullstellensatz (see, e.g.~\cite[Thm.1.5 (Ch.IX)]{Lang}), there exists 
$m\geq 1$ such that $(t_2)^m\in\ga$, the ideal generated by $r_2$, so $(t_2)^m$ is of the form
$r_2\cdot s$ for some polynomial $s$. Thus each factor of $r_2$ is a factor of $r_2\cdot s$, and
hence a factor of $(t_2)^m$. For the same reason, each factor of $r_2$ is one of $t_2$.
We may assume that the gcd of $t_2$'s coefficients is 1, 
and similar for $r_2$:
if not, then $t_2=k\cdot t'$ with $t'$ a polynomial with that property, and since
$k$ is a fixed numeral, we find $0\cdot k=0$ (also in fields with a characteristic 
that is a factor of $k$), and hence $0\cdot t_2=0\cdot t'$. 
By unique factorisation~(\cite[Cor.2.4 (Ch.IV)]{Lang}),  we find that $t_2$ and $r_2$
have equal primitive polynomials. 
Application of $\CMCZ$ 
(for the case $t_2=k\cdot t'$) and
equation~\eqref{e2} (that is, $0\cdot (x\cdot x)=0\cdot x$) then yields 
\begin{equation}
\label{v3}
\Mda+ \CMCZ\vdash 0\cdot t_2^{\inverse 1}= 0\cdot r_2^{\inverse 1}.
\end{equation}

\noindent
\underline{Ad~\eqref{p2}}. 
From Proposition~\ref{prop:3a} and validity of~\eqref{p2}
it follows that 
\begin{equation}\textstyle
\label{v4}
\Mda\vdash 0\cdot t_1+0\cdot t_2=0\cdot \sum_{x\in\VAR(t_1/t_2)}x
=0\cdot \sum_{x\in\VAR(r_1/r_2)}x=0\cdot r_1+0\cdot r_2.
\end{equation}

\noindent
\underline{Ad~\eqref{p3}}. We first derive 
\begin{align*}
\Mda\vdash 0\cdot t_1+0\cdot t_2^{\inverse 1}
&=0\cdot t_1+0\cdot(1+0\cdot t_2^{\inverse 1})\\
&=0\cdot t_1+0\cdot t_2\cdot t_2^{\inverse 1}
&&\text{by axiom~\eqref{Ma10}}\\
&=0\cdot t_1+0\cdot t_2+0\cdot t_2^{\inverse 1},
\end{align*}
and in a similar way one derives $\Mda\vdash 0\cdot r_1+0\cdot r_2^{\inverse 1}=
0\cdot r_1+0\cdot r_2+0\cdot r_2^{\inverse 1}$. 
Hence, we find with~\eqref{v3} and \eqref{v4}
that 
\begin{align}
\label{v5}
\Mda+ \CMCZ\vdash 
0\cdot t_1+0\cdot t_2^{\inverse 1}&=(0\cdot t_1+0\cdot t_2^{\inverse 1})+
(0\cdot r_1+0\cdot r_2^{\inverse 1})\\
\label{v6}
&=0\cdot r_1+0\cdot r_2^{\inverse 1}.
\end{align}

From $\CCM_0\models\eqref{p3}$ it follows from the completeness result on the 
class of involutive meadows of characteristic zero (see~\cite[Cor.2.7]{BBP2013a})
that $\Md+\CMCZ\vdash \eqref{p3}$, and also that 
$\Md + \CMCZ\vdash t_2 \cdot r_2\cdot(t_1 \cdot r_2 + (-r_1)\cdot t_2) =0$.
Because all coefficients in this identity are integer expressions, 
$\CMCZ$ plays no role in these proofs, so that  
$\Md \vdash t_2 \cdot r_2\cdot(t_1 \cdot r_2 + (-r_1)\cdot t_2) =0$. 
Moreover, $0\cdot s = 0$ can be applied to each coefficient $s$, so that 
$\Mda \vdash t_2 \cdot r_2\cdot(t_1 \cdot r_2 + (-r_1)\cdot t_2) =0\cdot(t_1\cdot t_2\cdot r_1\cdot r_2)$, 
from which one easily finds $\Mda\vdash \eqref{p3}$.

\medskip

Finally, we show the derivability of $t_1/t_2=r_1/r_2$ in $\Mda+\CMCZ$.
Multiplying both sides of~\eqref{p3} with 
$(t_2\cdot r_2)^{\inverse 1}\cdot (t_2\cdot r_2)^{\inverse 1}$ implies by~\eqref{e10}, 
$0\cdot x+0\cdot x=0\cdot x$, and equation~\eqref{e2} that 
\begin{align}
\nonumber
&\Mda\vdash
(t_2\cdot r_2)^{\inverse 1}\cdot(t_1\cdot r_2+ (-r_1)\cdot t_2)
+0\cdot t_2^{\inverse 1}+0\cdot r_2^{\inverse 1}
=
0\cdot t_1+0\cdot t_2^{\inverse 1}+0\cdot r_1+0\cdot r_2^{\inverse 1},
\intertext{which implies by Proposition~\ref{prop:2a} that}
\nonumber
&\Mda\vdash\frac{t_1}{t_2}+\frac{-r_1}{r_2}+0\cdot t_2^{\inverse 1}+0\cdot r_2^{\inverse 1}
=0\cdot t_1+0\cdot t_2^{\inverse 1}
+0\cdot r_1+0\cdot r_2^{\inverse 1},
\intertext{and thus }
\label{e:A}
&\Mda\vdash\frac{t_1}{t_2}+\frac{-r_1}{r_2}+0\cdot t_1+0\cdot t_2^{\inverse 1}
+0\cdot r_1+0\cdot r_2^{\inverse 1}=
0\cdot t_1+0\cdot t_2^{\inverse 1}
+0\cdot r_1+0\cdot r_2^{\inverse 1}.
\end{align}
Hence,
\begin{align*}
&\Mda+ \CMCZ\vdash\frac{t_1}{t_2}
=\frac{t_1}{t_2}+0\cdot t_1+0\cdot t_2^{\inverse 1}
\\
&~\hspace{21.6mm}
=\frac{t_1}{t_2}+0\cdot t_1 +0\cdot t_2^{\inverse 1}+0\cdot r_1+0\cdot r_2^{\inverse 1}
&&\text{by \eqref{v5}\hspace{4mm}}
\\
&~\hspace{21.6mm}
=\frac{t_1}{t_2}+(\frac{r_1}{r_2}+\frac{-r_1}{r_2})
+0\cdot t_1 +0\cdot t_2^{\inverse 1}+0\cdot r_1+0\cdot r_2^{\inverse 1}
\hspace{15.2mm}
\\
&~\hspace{21.6mm}
=(\frac{t_1}{t_2}+\frac{-r_1}{r_2})+\frac{r_1}{r_2}
+0\cdot t_1 +0\cdot t_2^{\inverse 1}+0\cdot r_1+0\cdot r_2^{\inverse 1}
\\
&~\hspace{21.6mm}
=\frac{r_1}{r_2}
+0\cdot t_1 +0\cdot t_2^{\inverse 1}+0\cdot r_1+0\cdot r_2^{\inverse 1}
&&\text{by \eqref{e:A}}\\
&~\hspace{21.6mm}
=\frac{r_1}{r_2}+0\cdot r_1+0\cdot r_2^{\inverse 1}
&&\text{by \eqref{v6}}\\
&~\hspace{21.6mm}
=\frac{r_1}{r_2}.
\end{align*}
\end{proof}

\section{Concluding Remarks}
\label{sec:4}

\paragraph{\bf Open Question.}
It is an open question whether there exists a basis result for the equational 
theory of $\CCM$.
We notice that in~\cite{BM14} a basis result for one-totalized non-involutive 
cancellation meadows is provided,
where the multiplicative inverse of 0 is 1 and cancellation is defined as usual 
(that is, by the cancellation law~\ref{eq:CL} in Section~\ref{subsec:versus}).

\paragraph{\bf Common Intuitions and Related Work.}
Common meadows are motivated as being the most intuitive modelling of a 
totalized inverse function to the best of our knowledge.
As stated in Section~\ref{sec:Intro} (Introduction), the use of the constant 
$\bot$ is a matter of convenience only because it merely 
constitutes a derived constant with defining equation $\bot = 0^{-1}$, which 
implies that all uses of $\bot$ can be removed.\footnote{%
  We notice that $0 = 1 + (-1)$, from which it follows that $0$ can also be 
  considered a derived constant over a reduced signature. Nevertheless, the 
  removal of $0$ from the signature of fields is usually not considered helpful.}
We notice that considering $\bot = 0^{-1}$ as an error-value supports the intuition 
for the equations of $\Mda$.

As a variant of involutive and common meadows, \emph{partial meadows} are 
defined in~\cite{BM2011a}. The specification method used in this paper
is based on meadows and therefore it is more simple, but
less general than the construction of Broy and Wirsing~\cite{BW81} for the specification of
partial datatypes. 

The construction of common meadows is related to the construction of \emph{wheels}
by Carlstr\"om~\cite{Carl}. 
However, we have not yet found a structural 
connection between both constructions which differ in quite important details. 
For instance, wheels are involutive whereas common meadows are non-involutive.

\paragraph{\bf Quasi-Cancellation Meadows of Characteristic Zero.}
Following Theorem~\ref{thm:basis},
a common meadow of characteristic zero can alternatively be defined as a structure 
that satisfies all equations true of all common cancellation meadows of 
characteristic zero. 
We write $\CM_0$ for the class of all common meadows of characteristic zero. 

With this alternative definition in mind, we define a 
\emph{common quasi-can\-cellation meadow of characteristic zero}
as a structure that satisfies all conditional equations which are true of all 
common cancellation meadows of characteristic zero. 
We write $\CQCM_0$ for the class of all common quasi-cancellation meadows of 
characteristic zero.
 
It is easy to show that $\CQCM_0$ is strictly larger than $\CCM_0$. 
To see this one extends the signature of common meadows 
with a new constant $c$. Let $L_{ccm,0}$ be the set of conditional equations 
true of all structures in $\CCM_0$. 
We consider the initial algebra of $L_{ccm,0}$ in the signature extended with $c$. 
Now neither $L_{ccm,0} \vdash c= \bot$ can hold 
(because $c$ might be interpreted as say 1), nor  $L_{ccm,0} \vdash 0 \cdot c = 0$ 
can hold (otherwise $L_{ccm,0} \vdash 0 = 0\cdot \bot = \bot$ would hold).
For that reason in the initial algebra of $L_{ccm,0}$ in the extended signature 
interprets $c$ as an entity $e$ in such a way  that neither
$c = \bot$  nor $0 \cdot c = 0$ is satisfied. For that reason
$c$ will be interpreted by a new entity that refutes $\CIL$.

$\CM_0$ is strictly larger than $\CQCM_0$. 
To see this let $E_{ccm,0}$ denote the set of equations valid in all common 
cancellation meadows of characteristic zero.  
Again we add an extra constant $b$ to the signature of common meadows. 
Consider the initial algebra $I$ of
$E_{ccm,0}+ (b^{-1} =\bot)$ in the extended signature. 
In $I$ the interpretation of $b$ is a new object because it cannot be 
proven equal to 0 and not to $\bot$ and not to any other closed term over the 
signature of common meadows. 
Now we transform $E_{ccm,0} + (b^{-1} =\bot)$ into its set of closed consequences 
$E_{ccm,0}^{cl,b}$ over the extended signature. 
We claim that $b = 0 \cdot b$ cannot be proven from $E_{ccm,0} + (b^{-1} =\bot$). 
If that were the case at some stage in the derivation an $\bot$ must appear from 
which it follows that $b=\bot$ is provable as well, because $\bot$ is propagated by 
all operations. 
But that cannot be the case as we have already concluded that $b$ differs from 
$\bot$ in the initial algebra $I_0$ of $E_{ccm,0}^{cl,b}$. 
Thus, $b\ne\bot\rightarrow  0 \cdot b=0$ (an instance of $\NVL$) is not valid in 
$I_0$.

However, at this stage we do not know the answers to the following two questions:
\begin{itemize}
\setlength{\itemsep}{0mm}
\item Is there a finite equational basis for the class $\CM_0$ of common meadows 
of characteristic zero?
\item Is there a finite conditional equational basis for the class $\CQCM_0$ of 
common quasi-cancellation meadows of characteristic zero?
\end{itemize}

\paragraph{\bf The Initial Common Meadow.}
In~\cite{BP14} we introduce \emph{fracpairs} with a definition that
is very close to that of the field of fractions of an integral 
domain. Fracpairs are defined over a commutative ring $R$ 
that is \emph{reduced}, i.e., $R$ has no nonzero nilpotent 
elements. A fracpair over $R$ is an expression $\dfrac pq$ with $p,q\in R$
(so $q=0$ is allowed) modulo the equivalence generated by
\begin{align*}
	 \frac{x \cdot z }{y \cdot (z \cdot z) } &= \frac{x}{y \cdot z}.
\end{align*}
This rather simple equivalence appears to be a congruence with respect to the 
common meadow signature $\Sigma_{\md,\bot}$
when adopting natural definitions:
\begin{align*}
&0=\frac 01, \quad 1=\frac 11, \quad \bot=\frac 10, 
\quad
\lh\frac pq\rh+\lh\frac rs\rh=
\frac{p\cdot s+r\cdot q}{q\cdot s}, 
\\
&\lh\frac pq\rh\cdot\lh\frac rs\rh=\frac{p\cdot r}{q\cdot s},
\quad
-\lh\frac pq\rh=\frac{-p}q, 
\quad\text{and}\quad
\lh\frac pq\rh^{-1}=\frac{q\cdot q}{p\cdot q}.
\end{align*}

In~\cite{BP14} we prove that the initial common meadow is isomorphic to
the initial algebra of fracpairs over the integers $\Int$.\footnote{%
  It should be mentioned that $\Mda$ as presented in~\cite{BP14}
  differs from the current version, as 
  explained in Footnote~\ref{vn} (page~\pageref{pagekey}).
  However, this makes no difference for the isomorphism 
  result mentioned: with both versions of $\Mda$ it easily follows that
  for every closed term $p$ over $\Sigma_f$, $0.p=0$, which is the 
  essential property.}
Moreover, we prove that the initial algebra of fracpairs over $\Int$ constitutes 
a homomorphic pre-image of the common meadow $\Rat_\bot$, 
and we define ``rational fracpairs'' over $\Int$ that constitute
an initial algebra that is isomorphic to $\Rat_\bot$. Finally, we consider 
some term rewriting issues for meadows.

These results reinforce our idea that common meadows can be used 
in the development of alternative foundations of elementary (educational)
mathematics from a perspective of abstract datatypes, term rewriting and 
mathematical logic.

\paragraph{Acknowledgement.}
We thank Bas Edixhoven (Leiden University) for helpful comments concerning the 
proof of Theorem~\ref{thm:basis},
including his suggestion to use~\cite{Lang} as a reference for this proof.
Furthermore, we thank two anonymous reviewers for valuable comments.
Finally, we thank John Tucker (Swansea) for valuable discussions 
that gave rise to re(formulating) axiom~\eqref{Ma12} of Table~\ref{Mda}
in this version (v4).

\end{document}